\documentclass[12pt, a4paper]{amsart}

\usepackage[unicode]{hyperref}
\hypersetup{colorlinks,bookmarksopen,bookmarksnumbered,citecolor=blue,pdfstartview=FitH}

\usepackage{color}

\usepackage[english]{babel}
\usepackage{amssymb,amsmath,amsthm}
\usepackage{verbatim}

\usepackage[all]{xy}

\newtheorem{dummy}{dummy}

\newtheorem{theorem}[dummy]{Theorem}

\newtheorem{remark}[dummy]{Remark}

\def\A{\mathbb A}

\def\P{\mathbb P}

\def\AA{\mathcal A}
\def\BB{\mathcal B}
\def\CC{\mathcal C}

\def\={\;=\;}
\def\bal{\begin{aligned}}
\def\eal{\end{aligned}}
\def\be{\begin{equation}\label}
\def\ee{\end{equation}}

\newcommand{\Sym}{{\rm Sym}}
\newcommand{\Spec}{{\rm Spec}}

\title{On a zeta function of a dg-category}

\author{Sergey Galkin,
Evgeny Shinder}

\begin{document}

\maketitle

\begin{abstract}

We define a zeta-function of a pretriangulated dg-category and investigate
its relationship with the motivic zeta function in the geometric case.

%

\end{abstract}


For a variety $X/k$ 
Kapranov's motivic zeta-function takes values in the Grothendieck ring $K_0(Var/k)$ of varieties over $k$ 
and is defined as
\be{def-mot}
Z_{mot}(X, t) = \sum_{n \ge 0} [\Sym^n(X)] \, t^n.
\ee

By construction it satisfies
\be{mult-kap}
Z_{mot}(X, t) = Z_{mot}(Z, t) \cdot Z_{mot}(U,t)
\ee
for any closed $Z \subset X$ with open complement $U \subset X$.

\medskip

In this short note we introduce another zeta-function, a zeta-function of a pre-triangluated 
dg-category and propose a relation between the two zeta-functions.
Everywhere in this paper $k$ is a field of characteristic zero.
Some of the technical proofs are omitted, details will be given elsewhere.

\medskip

We denote by the Grothendieck ring $K_0(dg-cat/k)$ 
the free abelian group generated by quasi-equivalence 
classes of pretriangulated dg-categories $\CC$
modulo relations
\[
[\CC] = [\AA] + [\BB]
\]
for semi-orthogonal decompositions $\CC = \langle \AA, \BB \rangle$.
See \cite[Section 5]{BLL} for details and the definition of the product on this ring.
Furthermore, symmetric power operations defined by Ganter-Kapranov \cite{GK}
give rise to $\lambda$-operations on $K_0(dg-cat/k)$.

This allows us to define the categorical zeta-function as
\[
Z_{cat}(\CC, t) = \sum_{n \ge 0} [\Sym^n(\CC)] \, t^n \in K_0(dg-cat/k)[[t]].
\]

By construction $Z_{cat}$ is multiplicative for semi-orthogonal decompositions:
\be{mult-cat}
Z_{mot}(\langle \AA, \BB \rangle, t) = Z_{cat}(\AA, t) \cdot Z_{cat}(\BB,t). 
\ee

\medskip

Recall that there exists a motivic measure
\[
\mu_{dg}: K_0(Var/k) \to K_0(dg-cat/k) 
\]
constructed in \cite[Section 8.2]{BLL} by sending a class of a smooth projective variety $X/k$
to the dg-enhancement $I(X)$ of the derived category of coherent sheaves on $X$.

The motivic measure $\mu_{dg}$ is a ring homomorphism, but it is \emph{not} a $\lambda$-ring homomorphism.
To illustrate the issue, consider the case of a point $X = \Spec(k)$.
Then the motivic zeta-function of $X$ is
\[
Z_{mot}(\Spec(k), t) = \frac1{1-t}, 
\]
whereas the categorical zeta-function of $\CC = \mu_{dg}(\Spec(k)) = I(Vect/k)$ counts irreducible representations of 
symmetric groups (recall that $char(k) = 0)$, i.e. is the partition function
\be{zeta-pt}
Z_{cat}(Vect/k, t) = \prod_{k \ge 1} \frac1{1-t^k}.
\ee

Similar relation holds in higher generality:

\begin{theorem}
Let $X$ be a smooth projective variety. Assume that either $dim(X) \le 2$ or that the
class of $X$ in the Grothendieck ring of varieties is a polynomial in $[\A^1]$.
Then there is a following relation between the zeta-functions: 
\be{zeta-exp}
Z_{mot}(\mu_{dg}(X),t) = \prod_{k \ge 1} \mu_{dg} (Z_{mot}(X,t^k)). 
\ee
\end{theorem}
\begin{proof}
The case of a point $X = \Spec(k)$ is (\ref{zeta-pt}).
Since 
\[
\mu_{dg}[\A^k] = \mu_{dg}[\A^1]^k = 1 
\]
the case of $[X]$ being a polynomial in $[\A^1]$ follows.

The case of a curve follows from the work of Polishchuk and van den Bergh \cite[Theorem 6.3.1]{PB}.
Finally the surface case is a combination of results of G\"ottsche \cite{G}, 
McKay correspondence of Bridgeland-King-Reid \cite{BKR} and Haiman \cite{H}.
\end{proof}

\begin{remark}
Since $Z_{cat}(X \times \P^n, t) = Z_{cat}(X, t)^{n+1}$, if the corresponding expansion 
(\ref{zeta-exp})
is known for all smooth projective varieties of
dimension $d$, then it holds for all dimensions $d' \le d$ at least in $K_0(dg-cat/k)[\frac1{d!}]$.

This way results of G\"ottsche \cite{G}, 
McKay correspondence of Bridgeland-King-Reid \cite{BKR} and Haiman \cite{H} 
which are used to prove (\ref{zeta-exp}) for surfaces 
imply a (very) weak version of \cite[Theorem 6.3.1]{PB}:
\[
[I_{S_n}(C^n)] = \sum_\alpha [I(C^{(\alpha)})] \in K_0(dg-cat/k) [\frac12]
\]
where $C$ is a smooth projective curve, the summation goes over partitions
$\alpha$ of $n$ and $C^{(\alpha)} = C^{(\alpha_1)} \times \dots C^{(\alpha_n)}$
for $\alpha = (1^{\alpha_1}, \dots, n^{\alpha_n})$.
\end{remark}

\begin{remark}
The transformation 
\[
f(t) \mapsto \prod_{k \ge 1} f(X,t^k) 
\]
of power series with constant term $1$ is invertible. Its inverse is given by
\[
g(t) \mapsto \prod_{k \ge 1} g(X,t^k)^{\mu(k)}
\]
where $\mu(k)$ is the M\"obius function. This fact follows easily from the M\"obius inversion formula.
\end{remark}

\providecommand{\arxiv}[1]{\href{http://arxiv.org/abs/#1}{\tt arXiv:#1}}

\newpage

\address{
{\bf Sergey Galkin}\\
National Research University Higher School of Economics (HSE)\\
Faculty of Mathematics and Laboratory of Algebraic Geometry\\
7 Vavilova Str. \\
117312, Moscow, Russia\\
e-mail: {\tt Sergey.Galkin@phystech.edu}
}

\medskip

\address{
{\bf Evgeny Shinder}\\
School of Mathematics and Statistics \\
University of Sheffield \\
The Hicks Building \\
Hounsfield Road \\
Sheffield S3 7RH\\
e-mail: {\tt eugene.shinder@gmail.com}
}


\begin{thebibliography}{xxx}


\bibitem[BLL04]{BLL}
Alexey I. Bondal, Michael Larsen, Valery A. Lunts:
\emph{Grothendieck ring of pretriangulated categories},
International Mathematics Research Notices, 2004(29), 1461--1495,

\bibitem[BKR99]{BKR}
Tom Bridgeland, Alastair King, Miles Reid:
\emph{The McKay correspondence as an equivalence of derived categories},
J. Amer. Math. Soc. 14 (2001), 535--554,

\bibitem[GK14]{GK}
N.\,Ganter, M.\,Kapranov,
\emph{Symmetric and exterior powers of categories},
Transform. Groups 19 (2014), no. 1, 57--103. 
 
\bibitem[G01]{G}
L.\,G\"ottsche:
\emph{On the motive of the Hilbert scheme of points on a surface},
Math. Res. Lett. 8 (2001), no. 5-6, 613--627. 

\bibitem[H01]{H}
M.\,Haiman:
\emph{Hilbert schemes, polygraphs and the Macdonald positivity conjecture},
J. Amer. Math. Soc., 14(4):941–1006 (electronic), 2001.

\bibitem[PvdB15]{PB}
A.\,Polishchuk, M.\,Van den Bergh:
\emph{Semiorthogonal decompositions of the categories of equivariant coherent sheaves for some reflection groups}, 
\arxiv{1503.04160}.


\end{thebibliography}
\end{document}